\documentclass[11pt]{article}

\usepackage[utf8]{inputenc}
\usepackage{amsmath,amssymb,latexsym, color}
\usepackage[normalem]{ulem}

\usepackage{amsthm}

\newtheorem{teor}{Theorem}

\newtheorem{defi}{Definition}

\newtheorem{prop}{Proposition}
\newtheorem{remark}{Remark}[section]

\newcommand{\R}{\mathbb{R}}

\newcommand{\M}{\mathcal{M}}

\title{Matrix Pearson equations satisfied by Koornwinder weights in two variables}

\author{Francisco Marcell\'an$^{a,}$\thanks{Partially supported by Ministerio de Econom\'ia y Competitividad of Spain, grant MTM2012--36732--C03--01.}
, Misael E. Marriaga$^{b}$, Teresa E. P\'{e}rez$^{c,}$\thanks{Partially supported by MICINN of Spain
        and by the European Regional Development Fund
        (ERDF) through the grant MTM2011--28952--C02--02, and Junta de
        Andaluc\'{\i}a FQM--0229 and P11--FQM--7276.}
, Miguel A. Pi\~nar$^{c,\dagger}$\\
\\ \small{$^{a}$ Instituto de Ciencias Matem\'aticas (ICMAT) and} \\
\small{Departamento de Matem\'aticas,}\\
\small{ Universidad Carlos III de Madrid (Spain)} \\
~~\\
\small{$^b$ Departamento de Matem\'aticas,}\\
\small{ Universidad Carlos III de Madrid (Spain)} \\
~~\\
\small{$^{c}$ Departamento de
 Matem\'{a}tica Aplicada,}\\
\small{Facultad de Ciencias}\\
 \small{Universidad de Granada (Spain)}}

\date{}

\begin{document}

\maketitle

\medskip


\medskip


\begin{abstract}

We consider Koornwinder's method for constructing orthogonal polynomials in two variables from
orthogonal polynomials in one variable. If semiclassical orthogonal polynomials in one variable
are used, then Koornwinder's construction generates semiclassical orthogonal polynomials in two
variables. We consider two methods for deducing matrix Pearson equations for weight functions
associated with these polynomials, and consequently, we deduce the second order linear partial
 differential operators for classical Koornwinder polynomials.
\end{abstract}


\section{Introduction}

In 1975, T. Koornwinder (\cite{Koor75}) introduced a non--trivial method to generate orthogonal
polynomials in two variables using univariate classical Jacobi polynomials. In fact, he studied
some classes of {\it two--variable analogues of the classical orthogonal polynomials}, and he
proved that all of these classes are eigenfunctions of second order linear partial differential operators.

Recently, in \cite{FPP12}, the authors studied the Koornwinder's construction in a more general
framework. They deduced some additional properties for weight functions associated with this
polynomials and introduced some new examples of bivariate Koornwinder polynomials.

\bigskip

Univariate {\it semiclassical} orthogonal polynomials were introduced for the first
time by E. Hendriksen and H. van Rossum in \cite{HR85} as the natural generalization of the classical
orthogonal polynomials. A weight function $w(x)$ defined over a
bounded or unbounded interval $(a,b)$ is said to be semiclassical if and only if it satisfies the Pearson equation
\begin{equation}\label{Uni-Pearson}
\frac{d}{dx} (\phi(x)\, w(x) ) = \psi(x)\,w(x),
\end{equation}
where $\phi(x)$ and $\psi(x)$ are fixed polynomials with $\deg\phi = p\ge 0$ and $\deg\psi=q\ge 1$,
respectively, and the boundary conditions
\begin{equation}\label{Uni-boundary}
\lim_{x\to a} \phi(x)\,w(x)\,p(x) = \lim_{x\to b} \phi(x)\,w(x)\,p(x) = 0,
\end{equation}
for every polynomial $p(x)$. Of course, the polynomials $\phi(x)$ and $\psi(x)$ in \eqref{Uni-Pearson} are not unique. This fact motivates the definition of \textit{class} of a semiclassical weight function introduced by P. Maroni in \cite{maroni1987} (see also \cite{maroni1991}). The class $s$ of a weight function $w(x)$ is defined as
\begin{equation}
s=\min \max \{\deg(\phi)-2,\deg(\psi)-1\},
\end{equation}
where the minimum is taken over all the polynomials $\phi$ and $\psi$ such that $w(x)$ satisfies the Pearson equation \eqref{Uni-Pearson}.

In \cite{maroni1991}, the author also proved that
orthogonal polynomials associated with semiclassical weight functions satisfy the difference--differential
equation
\begin{equation}\label{diff-diff}
\mathcal{L}[p_n] \equiv \phi(x) \,p_n''(x) + \psi(x)\, p'_n(x) = \sum_{i=n-s}^{n+s} \lambda_{n,i}\,p_i(x),
\quad n\ge s,
\end{equation}
where $s$ denotes the class of $w(x)$.

\bigskip

Naturally, the case $s=0$ reduces to the classical weight functions in one variable. Notice that
in this case \eqref{diff-diff} reads
$$
\mathcal{L}[p_n] \equiv \phi(x) \,p''_n(x) + \psi(x)\, p'_n(x) = \lambda_n\,p_n,\quad n\ge 0,
$$
where $\phi(x)$ and $\psi(x)$ are fixed polynomials with $\deg\phi \le 2$ and $\deg\psi= 1$, and
$\lambda_n\neq 0, n \ge 1$. Thus, the associated orthogonal polynomials are
eigenfunctions of the second order linear differential operator
$$
\mathcal{L}[\cdot] \equiv \phi(x) \,\frac{d^2}{dx^2} + \psi(x)\, \frac{d}{dx}.
$$
Bochner  (\cite{Bo29}) proved that Jacobi, Laguerre and Hermite orthogonal polynomials are the only
families of univariate orthogonal polynomials satisfying the above differential equation.

\bigskip

Classical and semiclassical weight functions in two variables can be defined by means of a bivariate
extension of the above definitions. In that case, the Pearson equation becomes
a matrix Pearson equation with matrix polynomial coefficients, and the derivative is replaced by
the usual divergence operator
$$\textnormal{div}\,(\Phi\, w(x,y)) = \Psi\,w(x,y),$$
where $\Phi$ is a $2\times 2$ symmetric polynomial matrix, and $\Psi$ is a $2\times 1$ polynomial vector,
as we will study in Section \ref{semi-section}.

The symmetric character of the matrix $\Phi$ is connected with the fact that orthogonal polynomials
associated with a semiclassical weight function $w(x,y)$ satisfy a difference--differential equation
whose coefficients are the entries of the matrices $\Phi$ and $\Psi$. In the classical case (when the
degrees of the entries of the matrices are less than or equal to $2$ and $1$, respectively), the
difference--differential equation becomes a partial differential equation for the orthogonal polynomials.

\bigskip

In this work, we study bivariate Koornwinder weight functions constructed from semiclassical
univariate weights. In this case, it was proved in \cite{FPP12} that $w(x,y)$ satisfies a matrix partial differential equation
$$\textnormal{div}(\varphi\,w(x,y)) = \delta\,w(x,y),$$
but the $2\times 2$ matrix $\varphi$ is not symmetric in general, and some of its entries
can be rational functions.

The main goal of this paper is to transform the above equation into a matrix Pearson equation
symmetrizing the matrix $\varphi$ in order to obtain a symmetric matrix $\Phi$, in such a way
that all the entries will be polynomials with the lowest possible degree.

\bigskip

The structure of this work is as follows. Section 2 presents some basic background about orthogonal
polynomials in two variables. Section 3 is focused on the basic background  on semiclassical and classical orthogonal polynomials in two variables. Koornwinder's method for
constructing systems of orthogonal polynomials in two variables as well as the construction of
semiclassical orthogonal polynomials in two variables with
Koornwinder's method are described in Section 4. In Section 5 we analyze two methods for finding
Pearson equations
for semiclassical and classical Koornwinder weights, and finally in Section 6 we provide examples
of these two methods and write second order linear partial differential operators associated with
semiclassical Koornwinder polynomials.

\bigskip


\section{Orthogonal polynomials in two variables}\label{section5}

Some background on orthogonal polynomials in two variables is introduced in this section for its use
throughout this work. We follow mainly \cite{DX2014}.

For $n\ge 0$, let $\Pi_n$ denote the linear space of real polynomials in two variables of
total degree not
greater than $n$, where the total degree of a polynomial is the highest combined degree of its
monomial terms. Let $\Pi=\bigcup_{n\ge 0} \Pi_n$ denote the linear space of all bivariate real polynomials. Observe that
$$\dim \Pi_n = \binom{n+2}{n},$$
and, for $n\ge 0$, there exist $n+1$ bivariate independent polynomials of exact degree $n$.

\medskip

Let $\M_{h \times k}(\R)$ denote the linear space of real matrices of size $h\times k$ and $\M_h(\R)$ denotes the space of real square matrices. Given a matrix $M\in \M_{h \times k}(\R)$, we denote by $M^t$ its transpose, and if $h=k$, $\det (M)$ denotes its determinant, and we say that $M$ is {\it non--singular}
if $\det(M)\neq 0$. The linear spaces of polynomial matrices and polynomial square matrices will be denoted by $\M_{h \times k}(\Pi)$ and $\M_h(\Pi)$, respectively. The degree of a polynomial matrix is defined as the maximum of the degrees of its polynomial entries. In addition, let $I_n$ denote the identity matrix of order
$n$.

\bigskip

Let $\Omega\subseteq \mathbb{R}^2$ be a domain having a non--empty interior. Suppose that
$w(x,y)$ is a non--negative
and integrable function defined on $\Omega$ such that
$$
\iint_{\Omega}w(x,y)dxdy>0,
$$
and the moments
$$
\mu_{h,k}=\iint_{\Omega} x^h\,y^k\,w(x,y)\,dxdy,
$$
are finite for all $h,k\ge0$, whether $\Omega$ is bounded or unbounded. Then $w(x,y)$ is said to be a
weight function over $\Omega$.

In this way, we can define the inner product
$$\langle p, q\rangle = \iint_{\Omega} p(x,y)\,q(x,y)\,w(x,y)\,dxdy,
$$
for all $p, q \in \Pi$. We say that $p\in \Pi_n$ is an {\it orthogonal polynomial with respect to $w(x,y)$} if
$$
\langle p, q\rangle \equiv 0, \qquad \forall q\in\Pi_{n-1}.
$$
Following \cite{DX2014}, let  denote by $\mathcal{V}_n$ the space of orthogonal polynomials of
degree exact $n$, that is,
$$\mathcal{V}_n = \{p\in \Pi_n: \langle p,q\rangle =0, \quad \forall q\in\Pi_{n-1}\}.
$$
Obviously, for $n\ge 1$, $\dim \mathcal{V}_n = n+1$, and we will denote an orthogonal basis of
$\mathcal{V}_n$ as $P_{n,k}(x,y), \,\, 0\le k\le n$. Observe that $\{P_{n,k}(x,y): 0\le k\le n,
\, n\ge 0\}$ is a sequence of independent bivariate polynomials
such that
\begin{itemize}
\item $\deg P_{n,k} = n,\quad n\ge 0, \quad 0\le k \le n$
\item $\langle P_{n,k}, P_{m,j}\rangle = K_{n,k}\,\delta_{n,m}\,\delta_{k,j}, \quad K_{n,k}> 0$.
\end{itemize}
Then, we will call it a {\it sequence of bivariate orthogonal polynomials} associated with the
weight function $w(x,y)$.

\bigskip

Suppose that $f:\mathbb{R}^2\rightarrow\mathbb{R}$ and $\mathbf{F}:\mathbb{R}^2\rightarrow\mathbb{R}^2$. In this work, the {\it gradient operator}
$$\nabla f(x,y)=(\partial_x f,\,\partial_y f)^t,$$
and the {\it divergence operator}
$$\textnormal{div}\,\mathbf{F}(x,y)=\nabla \cdot \mathbf{F},$$
will be used as the standard differential operators in two variables.

\bigskip


\section{Classical and semiclassical weight functions in two variables}\label{semi-section}

The contents of this section are dedicated to recall the definition of the classical and
semiclassical character for weight functions in two variables (\cite{AdMFPn07, AdMFPn08}).

\begin{defi}
Let $w(x,y)$ be a bivariate weight function defined over the domain $\Omega$. Then $w(x,y)$ is said
to be {\it semiclassical} if there exist a non--zero symmetric polynomial matrix and a non--zero
polynomial vector
\begin{equation}\label{phipsi}
\Phi=\begin{pmatrix}
\phi_{1,1} & \phi_{1,2} \\
\phi_{1,2} & \phi_{2,2}
\end{pmatrix}\in\mathcal{M}_{2\times2}(\Pi), \quad
\Psi=\begin{pmatrix}
\psi_1 \\
\psi_2
\end{pmatrix}\in\mathcal{M}_{2\times1}(\Pi),
\end{equation}
with $\textnormal{deg}\,\Phi\ge0$ and $\textnormal{deg}\,\Psi\ge1$, such that $\det\langle 1,\Phi
\rangle \ne 0$ and $w(x,y)$ satisfies the matrix Pearson equation
\begin{equation}\label{pearsontypeq1}
\textnormal{div}(\Phi w)= \Psi^t\, w,
\end{equation}
and the boundary conditions
\begin{equation}\label{boundaryconditions}
\begin{array}{l}
 \displaystyle{\int_{\partial \Omega}} p(x,y)\,w(x,y)\,(\phi_{1,1}(x,y)dy-\phi_{1,2}(x,y)dx)=0\\
 ~~\\
\displaystyle{\int_{\partial \Omega}} p(x,y)\,w(x,y)\,(\phi_{1,2}(x,y)dy-\phi_{2,2}(x,y)dx)=0,
\end{array}
\end{equation}
must hold for every polynomial $p(x,y)$. Moreover, we define
\begin{equation}\label{ese}
s = \max \{\deg(\Phi)-2,\deg(\Psi)-1\}.
\end{equation}
\end{defi}

\begin{remark}
The matrix Pearson equation for a given weight function is not unique. In fact,
 \eqref{pearsontypeq1} can be left multiplied times another $2\times 2$ non--singular
polynomial matrix, and we obtain a new matrix Pearson equation for $w(x,y)$.

The problem of characterize the minimum matrix Pearson equation for a semiclassical weight
function remains open.
\end{remark}

\medskip

From the definition of semiclassical weight function, we can consider a {\it classical} one as a particular case.

\begin{defi}
A weight function $w(x,y)$ defined on a domain $\Omega \subseteq \mathbb{R}^2$ is called {\it classical} if
it is semiclassical with $\deg \Phi\le 2$ and $\deg \Psi =1$. Using definition \eqref{ese}, we get that,
for classical weights, $s=0$.
\end{defi}

Classical and semiclassical weight functions are characterized by difference--differential properties
for the associated orthogonal polynomial sequences. Using the matrices defined in \eqref{phipsi},
we introduce the partial differential operator
$$\mathcal{L}[\cdot]:= \phi_{1,1}\,\partial_{xx} + 2\phi_{1,2}\,\partial_{xy} + \phi_{2,2}\,\partial_{yy}
+ \psi_{1}\,\partial_x + \psi_2\,\partial_y.
$$
We recall the next characterization for semiclassical weight functions.

\begin{teor}[\cite{AdMFPn08}]
Let $\{P_{n,k}(x,y): 0\le k\le n, \quad n\ge0\}$ be an orthogonal polynomial sequence associated with a
weight function $w(x,y)$. Then $w(x,y)$ is semiclassical, that is, it satisfies a Pearson equation
(\ref{pearsontypeq1}) if and only if for each $n\ge0, \, 0\le k\le n$, $P_{n,k}(x,y)$ satisfies the
second order difference--differential relation
\begin{equation}\label{diffrel2d}
\mathcal{L}[P_{n,k}(x,y)] = \sum_{m=n-s}^{n+s}\sum_{i=0}^m \lambda^{n,k}_{m,i} \,P_{m,i}(x,y),
\end{equation}
where $\lambda^{n,k}_{m,i}\in \mathbb{R}$.
\end{teor}

Observe that if $w(x,y)$ is classical, then $s=0$, and in the above difference--differential
relation the double sum in the right hand side reduces to a sum of orthogonal polynomials
in $\mathcal{V}_n$. This characterization for bivariate classical orthogonal polynomials was shown in \cite{FPP05a} and it can be reformulated in the following way.

\begin{teor}[\cite{FPP05a}]\label{teor-clas} In the above conditions, $w(x,y)$ is classical if and only if
there exists constants $\lambda^{n,k}_{i}\in \mathbb{R}$ such that
\begin{equation*}
\mathcal{L}[P_{n,k}(x,y)] = \sum_{i=0}^n \lambda^{n,k}_{i} \,P_{n,i}(x,y).
\end{equation*}

\end{teor}

\begin{remark} We must remark that in the classical case, the whole linear space of orthogonal
polynomials of exact degree $n$ is preserved by $\mathcal{L}$.
In other words,
\begin{equation*}
\mathcal{L}[\mathcal{V}_n] \subset \mathcal{V}_n.
\end{equation*}
\end{remark}

\noindent
In the particular case when $\lambda^{n,k}_{i}=\delta_{k,i} \lambda_{n,k}$, every polynomial of the above sequence 
is an eigenfunction of $\mathcal{L}$
$$
\mathcal{L}[P_{n,k}] = \lambda_{n,k}\,P_{n,k},
$$
that is, every orthogonal polynomial satisfies a second order linear partial differential equation.
Theorem \ref{teor-clas} is an extension of the Krall and Sheffer's definition of classical
orthogonal polynomials in two variables (\cite{KS67}). In that case,
$\lambda^n_k \equiv \lambda_n$ is independent of $k$, and then every orthogonal polynomial of
total degree $n$ satisfies the same second order linear partial differential equation.


\bigskip

\section{Two variable semiclassical Koornwinder weights}

First, we describe the method introduced by T. H. Koornwinder in 1975 (see \cite{DX2014,Koor75}),
to construct weight functions in two variables from two
weight functions in one variable.

\medskip

Let $w_1(x)$ and $w_2(y)$ be univariate weight functions defined on the intervals $(a,b)$ and $(c,d)$,
respectively. Let $\rho(x)$ be a positive function on $(a,b)$ satisfying one of the following two conditions

\medskip

\leftline{
\begin{tabular}{lp{0.85\textwidth}}
\emph{Case I}: & $\rho(x)$ is a polynomial of degree $\le 1$, that is, $\rho(x) = r_1\,x+r_0$,
with $|r_1|+|r_0|>0$,\\
& \\
\emph{Case II}: & $\rho(x)$ is the square root of a non--negative polynomial of degree at most 2,
$c=-d < 0$, and $w_2(y)$ is an even function on $(-d, d)$.
\end{tabular}}

\bigskip

Anyway, $\rho(x)^2$
is a polynomial of degree less than or equal to 2, and from now on,
we will denote
\begin{equation*}
\rho(x)^2 = a_2\, x^2 + a_1\, x + a_0,
\end{equation*}
where $a_2, a_1, a_0 \in \mathbb{R}$ and $|a_2|+|a_1| + |a_0|>0$. Observe that, in
the first case $a_2 = r_1^2\ge 0$, $a_1 = 2\,r_1\,r_0$, and $a_0 = r_0^2\ge 0$.

\bigskip

For $m\ge 0$, let $\{p_{n}(x;m)\}_{n\geqslant0}$ be the monic orthogonal polynomial sequence
with respect to the weight function $\rho(x)^{2m+1}\,w_1(x)$ and let $\{q_n(y)\}_{n\geqslant0}$
be the monic orthogonal polynomial sequence with respect to the weight function $w_2(y)$.
Then, we define the monic bivariate {\it Koornwinder polynomials}
\begin{equation}\label{kops}
P_{n,m}(x,y) = p_{n-m}(x;m)\,\rho(x)^m\, q_m\left(\frac{y}{\rho(x)}\right), \quad 0 \leqslant m\leqslant n.
\end{equation}
Notice that we get a polynomial of total degree $n$ and degree $m$ in $y$. Moreover, they are orthogonal with respect to the {\it Koornwinder weight function}
\begin{equation}\label{Koorw}
w(x,y)=w_1(x)\,w_2\left(\frac{y}{\rho(x)}\right),
\end{equation}
over the domain
\begin{equation}\label{domain}
\Omega = \{(x,y)\in \mathbb{R}^2,\quad  a<x<b, \quad c\,\rho(x) < y < d\,\rho(x)\}.
\end{equation}

\medskip

Observe that the \emph{tensor product} of two monic orthogonal polynomials in one
variable
$$P_{n,m}(x,y) = p_{n-m}(x)\,q_m(y), \quad 0 \leqslant m\leqslant n,
$$
corresponds to monic Koornwinder orthogonal polynomials with respect to the weight function $w(x,y) = w_1(x)\,w_2(y)$ where $\rho(x)=1$.


\bigskip

Now, we will recover a Theorem in \cite{FPP12} where it was proved that the semiclassical character is
inherited by bivariate Koornwinder polynomials

\begin{teor}[\cite{FPP12}]
Let $w_1$ and $w_2$ be two semiclassical weight functions in one variable. Then, the bivariate
Koornwinder weight \eqref{Koorw} is semiclassical.
\end{teor}

The proof of above Theorem does not provides an standard method to find a minimal (in the sense
of the degrees of the coefficients) matrix Pearson equation for Koornwinder weights. This is the
main objective from now on.

Let $w_1(x)$ and $w_2(x)$ be two semiclassical weight functions in one variable defined on $(a,b)$ and
$(c,d)$, respectively, and let
$$
\frac{d}{dx}(\phi_i(x)\,w_i(x)) = \psi_i(x)\,w_i(x), \quad i=1,2,
$$
its respective Pearson equations, where $\deg\,\phi_i = p_i \ge 0$ and $\deg\,\psi_i = q_i \ge1$.
The Pearson equations are equivalent to
$$
\phi_i(x) w'_i(x) = \widetilde{\psi}_i(x) w_i(x),
$$
where $\widetilde{\psi}_i(x) = \psi_i(x) - \phi'_i(x), \, i=1,2$.

Taking partial derivatives on \eqref{Koorw}, we get
\begin{align*}
&\frac{\partial}{\partial x}w(x,y) = w'_1(x)\,w_2\left(\frac{y}{\rho(x)}\right) - w_1(x)\,w'_2\left(\frac{y}{\rho(x)}\right)\frac{y}{\rho(x)^2}\,\rho'(x),\\
\\
&\frac{\partial}{\partial y}w(x,y) = w_1(x)\,w'_2\left(\frac{y}{\rho(x)}\right)\frac{1}{\rho(x)}.
\end{align*}
Then, substituting the second equation into the first one and using the Pearson equations for
$w_1$ and $w_2$, we deduce
\begin{align}
&\phi_1(x)\frac{\partial}{\partial x} w(x,y) + \phi_1(x) \frac{\rho'(x)}{\rho(x)}\,y\,
\frac{\partial}{\partial y} w(x,y) = \widetilde{\psi}_1(x)\,w(x,y),\label{first}\\
\nonumber \\
&\phi_2\left(\frac{y}{\rho(x)} \right) \frac{\partial}{\partial y} w(x,y) = \frac{1}{\rho(x)}
\widetilde{\psi}_2\left(\frac{y}{\rho(x)}\right)\,w(x,y).\label{second}
\end{align}

If we define the following matrices
\begin{equation*}
\varphi =\begin{pmatrix}
\varphi_1 & \varphi_2\\
\\
0 & \varphi_3
\end{pmatrix}= \begin{pmatrix}
\phi_1(x) & \eta(x)\,y\\
\\
0 & \rho(x)\phi_2\left(\dfrac{y}{\rho(x)} \right)
\end{pmatrix}, \quad \delta = \begin{pmatrix}
\delta_1 \\ \\  \delta_2
\end{pmatrix}= \begin{pmatrix}
\widetilde{\psi}_1(x) \\ \\  \widetilde{\psi}_2\left(\dfrac{y}{\rho(x)}\right)
\end{pmatrix},
\end{equation*}
equations \eqref{first}--\eqref{second} can be written as
\begin{equation}\label{eq30}
\varphi \,\nabla w = \delta \, w,
\end{equation}
where
\begin{equation*}
\eta(x)=\phi_1(x)\frac{\rho'(x)}{\rho(x)}.
\end{equation*}

\begin{remark} If both univariate weight functions are semiclassical and $\rho(x) =1$, Koornwinder construction
yields semiclassical weight functions in two variables since matrix $\varphi$ is diagonal. In particular,
tensor product of univariate classical weight functions provides a bivariate classical weight.
\end{remark}

We must remark that equation \eqref{eq30} is not a matrix Pearson equation for the Koornwinder
weight since the coefficient matrix $\varphi$ is not a symmetric matrix in general and it is not guaranteed that its entries are polynomials.

\bigskip

Observe that the determinant of $\varphi$ does not vanish on the interior of $\Omega$,
the domain of orthogonality for Koornwinder polynomials given in (\ref{domain}). In fact,
$$\det\,\varphi=\phi_1(x)\,\rho(x)\phi_2\left(\frac{y}{\rho(x)}\right)\neq 0,
$$
for all $(x,y)$ belonging to the interior of $\Omega$. Therefore, $\varphi$ is a non--singular
matrix on the interior  of $\Omega$ and we can solve (\ref{eq30}) obtaining
\begin{equation}\label{eq35}
\nabla\,w(x,y) = \varphi^{-1}\,\delta\,w(x,y),
\end{equation}
where
$$
 \varphi^{-1}\, \delta =\begin{pmatrix}
\dfrac{\widetilde{\psi}_1(x)}{\phi_1(x)}-\dfrac{\eta(x)y\widetilde{\psi}_2\left(\dfrac{y}{\rho(x)}
\right)}{\phi_1(x)\rho(x)\phi_2\left(\dfrac{y}{\rho(x)}\right)} \\
\\
\dfrac{\widetilde{\psi}_2\left(\dfrac{y}{\rho(x)}\right)}{\rho(x)\phi_2\left(\dfrac{y}{\rho(x)}\right)}
\end{pmatrix}.
$$
The entries of the column vector $\varphi^{-1}\delta$ are rational functions whose denominator can only
vanish on the set
\begin{equation}\label{set}
\left\{(x,y)\in\mathbb{R}^2:\,\det\,\varphi=\phi_1(x)\,\rho(x)\phi_2\left(\frac{y}{\rho(x)}\right)= 0\right\},
\end{equation}
and therefore, $w$ may only vanish or become infinite on \eqref{set}. Moreover, since $\nabla \ln w=\nabla w/w$, then the partial derivatives of all orders of $\ln w$ are rational functions whose denominators do not vanish outside  \eqref{set}. Hence, $\ln w$ is analytic, and consequently, so is $w$.

\bigskip

Now, we study when the rational function
$$\eta(x)=\phi_1(x)\,\frac{\rho'(x)}{\rho(x)},$$
is a polynomial. Notice that
$$
\frac{d}{dx}[\phi_1(x)\,\rho(x)^{2m+1}\,w_1(x)] =
\left[\psi_1(x)+(2m+1)\phi_1(x)\frac{\rho'(x)}{\rho(x)}\right]\rho(x)^{2m+1}\,w_1(x),
$$
that is, the weight function $u_m(x)=\rho(x)^{2m+1}\,w_1(x)$ satisfies
\begin{equation}\label{Pearson-m}
\frac{d}{dx}[\phi_m(x)\,u_m(x)] =
\psi_m(x)\,u_m(x),
\end{equation}
where
$$\phi_m(x) = \phi_1(x), \qquad \psi_m(x) = \psi_1(x)+(2m+1)\phi_1(x)\frac{\rho'(x)}{\rho(x)}.
$$
In order to have a Pearson equation for the weight function $u_m(x)$, we need that the coefficients
of \eqref{Pearson-m} to be polynomials.

\begin{prop}

If the weight function $w_1(x)$ is semiclassical of class $s_1$, then $u_m(x)$ is semiclassical of class
$c_m$, where
\begin{itemize}
\item \emph{Case I}: if $\rho(x)$ divides $\phi_1(x)$, then $c_m=s_1$, otherwise if $\rho(x)$ does not
divides $\phi_1(x)$, then $c_m=s_1+1$,

\item \emph{Case II}: if $\rho(x)^2$ divides $\phi_1(x)$, we get $c_m=s_1$, while
if $\rho(x)^2$ does not divides $\phi_1(x)$, then $s_1 + 1 \le c_m \le s_1+2$.

\end{itemize}
As a consequence, if $w_1(x)$ is classical and
\begin{itemize}
\item \emph{Case I}: $\rho(x)$ divides $\phi_1(x)$,

\item \emph{Case II}: $\rho(x)^2$ divides $\phi_1(x)$,

\end{itemize}
then $u_m(x)$ is classical, and the rational function $\eta(x)$ is a polynomial of degree $\le 1$.
\end{prop}

\begin{proof}

In {\it Case I}, $\rho(x) = r_1\,x+r_0$, with $|r_1|+|r_0|>0$.

If $\rho(x)$ divides $\phi_1(x)$, the
weight $u_m(x)$ is semiclassical of the same class as
$w_1(x)$. On the other hand, if $\rho(x)$ does not divide $\phi_1(x)$, then multiplying \eqref{Pearson-m}
times $\rho(x)$, we deduce that $u_m(x)$ is again semiclassical, but its class increases.

In {\it Case II}, we know that $\rho(x) = \sqrt{a_2\, x^2 + a_1\, x + a_0}$ and
$$\frac{\rho'(x)}{\rho(x)} = \frac{2\,a_2 \,x+ a_1}{2\,\rho(x)^2}.$$
Again, $u_m(x)$ is semiclassical, and if $\rho(x)^2$ divides $\phi_1(x)$, the weight function $u_m(x)$
is semiclassical of the same class as $w_1(x)$, but on the contrary the class increases.

\end{proof}


\bigskip

\section{Two symmetrization methods}

In this section we will explain two methods for symmetrizing (\ref{eq30}). The first method is based
on finding a matrix $S$ with rational entries such that the matrix $S\,\varphi$ is symmetric, that is,
such that the condition $S\varphi-\varphi^tS^t=0$ holds. Additionally, the entries of the matrix $
S\varphi$ and the vector $S\,\delta$ must be polynomials of the lowest possible degree. We choose to
call this first method {\it the matrix symmetrization method}. The second method consists on factorizing
elements of $\varphi$ and $\delta$, such as their entries and $\det\varphi$, and then using these
factorizations to construct auxiliary functions that will help turn (\ref{eq30}) into a matrix Pearson
equation for the weight $w$. This second method will be called {\it the decomposition method}. We will
make these descriptions of both methods in a more precise way in the sequel.

\bigskip

\subsection{The symmetrization method}

We want to symmetrize the matrix $\varphi$ by finding a matrix
\begin{equation}
S\equiv\begin{pmatrix}
A & B\\
C & D
\end{pmatrix},
\end{equation}
where $A=A(x,y),\,B=B(x,y),\,C=C(x,y),$ and $D=D(x,y)$ are rational functions such that $AD-BC\neq 0$
on the interior of $\Omega$. Then, it is required that the matrix $S$ left--multiplies $\varphi$ and
transforms it in a symmetric matrix, that is,
\begin{equation}\label{sym}
\begin{pmatrix}
A & B \\
C & D \end{pmatrix}\begin{pmatrix}
\varphi_1 & \varphi_2 \\
0 & \varphi_3 \end{pmatrix}-\begin{pmatrix}
\varphi_1 & \varphi_2 \\
0 & \varphi_3 \end{pmatrix}^T\begin{pmatrix}
A & B \\
C & D \end{pmatrix}^T=\begin{pmatrix}
0 & 0 \\
0 & 0 \end{pmatrix}.
\end{equation}

Notice that (\ref{sym}) yields the constraint that $A,\, B$, and $C$ must satisfy, namely,
\begin{equation}
A\varphi_2+B\varphi_3-C\varphi_1=0.
\end{equation}
This restriction admits polynomials and rational functions as solutions. Nevertheless, we must point out that polynomial solutions will always increase the degree of $\varphi$.

Furthermore, in order to have a matrix Pearson equation for the Koornwinder weight $w$ satisfying
$S\,\varphi\, \nabla w = S\,\delta \,w$, the vector
\begin{equation*}
\widetilde{\Psi}_{\delta}\equiv S\,\delta=\begin{pmatrix}\delta_1\,A+\delta_2\,B\\ \delta_1\,C+
\delta_2\,D \end{pmatrix},
\end{equation*}
must have polynomial entries. Define the matrix $\Phi_{\varphi}$ and the column vector $\Psi_{\delta}$ as
\begin{equation*}
\Phi_{\varphi}=S\,\varphi =\begin{pmatrix}
A\,\varphi_1 & A\,\varphi_2+B\,\varphi_3 \\
C\,\varphi_1 & C\,\varphi_2+D\,\varphi_3
\end{pmatrix}, \quad \Psi_{\delta}=\widetilde{\Psi}_{\delta}+(\textnormal{div}\,\Phi_{\varphi})^t.
\end{equation*}
Then, $w(x,y)$ satisfies $\textnormal{div}(\Phi_{\varphi} \,w)=\Psi_{\delta}^t\, w$. We must choose $S$ such that $\Phi_{\varphi}$ has minimal degree and $\textnormal{deg}\,\Psi_{\delta}\ge 1$.


\bigskip

\subsection{The decomposition method}

We can always find a common polynomial denominator $E=E(x,y)$ for the rational entries of the
column vector multiplying $w(x,y)$ in the right side of (\ref{eq35}). Note that $E(x,y)\ne0$ for
all $(x,y)$ in the interior of $\Omega$, because if $E(x_0,y_0)=0$ for some $(x_0,y_0)$ in the interior of $\Omega$, then $w(x,y)$ would not be defined at $(x_0,y_0)$.

From now on, we will write (\ref{eq35}) using this common polynomial denominator as
\begin{equation}\label{eq43}
\begin{pmatrix}
E & 0 \\
0 & E
\end{pmatrix}\nabla w= \begin{pmatrix}
F \\
H
\end{pmatrix}w,
\end{equation}\label{eq36}
\begin{equation}
E = a_0\,a_1\,c_1,\quad F = F_0\,c_1, \quad H = H_0\,a_1,
\end{equation}
where $F=F(x,y)$ and $H=H(x,y)$ are polynomials. Observe that (\ref{eq43}) is not necessarily the
desired Pearson equation for $w(x,y)$. We have allowed the possibility of $E$ having common
factors $c_1=c_1(x,y)$ and $a_1=a_1(x,y)$ with $F$ and $H$, respectively. There is no loss of
generality here since we can always get either $c_1$ or $a_1$, or both be equal to 1.

We seek polynomials $a_2=a_2(x,y),\,b_1=b_1(x,y)$, and $c_2=c_2(x,y)$ such that
\begin{equation*}
a_0=a_2c_2-a_1b_1^2c_1.
\end{equation*}

Let introduce the auxiliary functions
\begin{equation}\label{eq39}
a(x,y)=\frac{a_2(x,y)}{a_0(x,y)\,c_1(x,y)},\quad b(x,y)=\frac{b_1(x,y)}{a_0(x,y)},\quad c(x,y)
=\frac{c_2(x,y)}{a_0(x,y)\,a_1(x,y)},
\end{equation}
and the matrix
\begin{equation}\label{eq40}
\begin{pmatrix}
a & b\\
b & c
\end{pmatrix}, \quad ac-b^2=\frac{1}{E}.
\end{equation}

After left--multiplying (\ref{eq35}) by (\ref{eq40}), we get
\begin{equation}\label{eq41}
\begin{pmatrix}
a\,E & b\,E\\
b\,E & c\,E
\end{pmatrix}\nabla w=\begin{pmatrix}
a\,F+b\,H\\
b\,F+c\,H
\end{pmatrix}w.
\end{equation}
From the decomposition method, to find a Pearson equation for $w$ we must to obtain three polynomials $a_2,\,b_1$, and $c_2$ such that the matrix coefficient in (\ref{eq41}) has polynomial entries of lowest total degree possible, and the column vector on the right side of the same equation has polynomial entries.


\bigskip

\section{Examples}

In this section we will denote by $\{P_n^{(\alpha,\beta)}\}_{n\ge 0}$ the sequence of classical Jacobi polynomials associated with the weight function
\begin{equation*}
w^{(\alpha,\beta)}(x)=(1-x)^{\alpha}(1+x)^{\beta}, \quad -1\le x\le 1, \quad \alpha,\beta>-1
\end{equation*}
(see \cite{Ch78,Sz78}). The Pearson equation for Jacobi polynomials is 
\begin{equation*}
\frac{d}{dx}\left[(1-x^2)\,w^{(\alpha,\beta)}\right] =
[\beta-\alpha -(\alpha+\beta+2)x]w^{(\alpha,\beta)},
\end{equation*}
that is, $\phi(x) = 1-x^2$ and $\psi(x) = \beta-\alpha -(\alpha+\beta+2)x$.

\medskip

Classical Jacobi polynomials can be defined on the interval $[0,1]$. In this case, the weight function is given by
\begin{equation*}
u^{(\alpha,\beta)}(x)=(1-x)^{\alpha}x^{\beta}, \quad \alpha,\beta>-1,
\end{equation*}
and the Pearson equation for $u^{(\alpha, \beta)}$ is
\begin{equation*}
\frac{d}{dx}\left[(1-x)x\,u^{(\alpha,\beta)}\right]=[\beta+1 - (\alpha+\beta+2)x]u^{(\alpha,\beta)}.
\end{equation*}
In this case, $\phi(x) = (1-x)x$ and $\psi(x) = \beta+1 -(\alpha+\beta+2)x$.

\medskip

On the other hand, we will denote by $\{L^{(\alpha)}_n\}_{n\ge0}$ the sequence of classical Laguerre polynomials associated with the weight function
\begin{equation*}
w^{(\alpha)}(x)=x^{\alpha}e^{-x}, \quad 0\le x<\infty,\quad \alpha>-1,
\end{equation*}
whose Pearson equation is 
\begin{equation*}
\frac{d}{dx}\left[x\,w^{(\alpha)}\right] = [\alpha + 1 - x]w^{(\alpha)},
\end{equation*}
where $\phi(x) = x$ and $\psi(x) = \alpha + 1 - x$.


\subsection{Ball polynomials}

Let
\begin{equation*}
\mathbf{B}=\{(x,y)\in\mathbb{R} :\, x^2+y^2\le 1\},
\end{equation*}
be the unit disk in $\mathbb{R}^2$, and let
\begin{equation*}
w(x,y)=(1-x^2-y^2)^{\alpha}, \quad \alpha>-1,
\end{equation*}
be the weight function. Ball polynomials can be constructed by using Koornwinder's
method taking
\begin{align*}
&w_1(x)=(1-x^2)^{\alpha}, \quad -1\le x\le 1,\\
&w_2(y)=(1-y^2)^{\alpha}, \quad -1\le y\le 1,\\
&\rho(x)=\sqrt{1-x^2}.
\end{align*}
Then, ball polynomials can be defined as
\begin{equation*}
P_{n,m}(x,y)=P_{n-m}^{(\alpha+m+1/2,\alpha+m+1/2)}(x)\,(1-x^2)^{m/2}\,P_m^{(\alpha,\alpha)}
\left(\frac{y}{1-x^2}\right), \quad 0\le m\le n.
\end{equation*}
Observe that, in this case, $\phi_1(x) = \phi_2(x) = \rho(x)^2$, the weight function
\begin{equation*}
w(x,y)=w_1(x)w_2\left(\frac{y}{\rho(x)}\right)=(1-x^2-y^2)^{\alpha},\quad \alpha>-1,
\end{equation*}
satisfies (\ref{eq30}) where
\begin{equation}\label{eq42}
\varphi=\begin{pmatrix}
1-x^2 & -xy \\
0 & 1-x^2-y^2
\end{pmatrix}, \qquad \psi=\begin{pmatrix}
-2\alpha x \\ -2\alpha y\end{pmatrix}.
\end{equation}
A suitable choice for the symmetrization matrix of (\ref{eq42}) is
\begin{equation*}
S =
\begin{pmatrix}
1 & 0\\
\dfrac{-xy}{1-x^2} & \dfrac{1}{1-x^2}
\end{pmatrix},
\end{equation*}
and after a symmetrization using $S$, we recover the well known matrix Pearson equation for
ball weight
\begin{equation}\label{eq44}
\begin{pmatrix}
1-x^2 & -xy \\
-xy & 1-y^2
\end{pmatrix}\,\nabla w=\begin{pmatrix}
-2\alpha x \\ -2\alpha y
\end{pmatrix}\,w.
\end{equation}
The second order linear partial differential operator for ball polynomials is
$$
\mathcal{L}[\cdot]=
(1-x^2)\partial_{xx}-2xy\partial_{xy}+(1-y^2)\partial_{yy}-(2\alpha+3)x\partial_x
-(2\alpha+3)y\partial_y,
$$
and, therefore, ball polynomials satisfy the Krall and Sheffer second order
linear partial differential equation
$$\mathcal{L}[P_{n,m}] = -n(n+2\alpha+2)P_{n,m}.
$$
If the decomposition method is used, then a suitable choice of auxiliary functions (\ref{eq39}) are
\begin{equation*}
a(x,y)=\frac{1-x^2}{1-x^2-y^2},\quad b(x,y)=\frac{-xy}{1-x^2-y^2}, \quad c(x,y)=\frac{1-y^2}{1-x^2-y^2},
\end{equation*}
and we obtain again (\ref{eq44}). Another suitable choice of auxiliary functions is
\begin{equation*}
a(x,y)=1,\quad b(x,y)=0,\quad c(x,y)=1,
\end{equation*}
and we obtain another Pearson equation (see \cite{Le00}),
\begin{equation}
\begin{pmatrix}
1-x^2-y^2 & 0\\
0 & 1-x^2-y^2
\end{pmatrix}\nabla w=\begin{pmatrix}
-2\alpha x \\ -2\alpha y
\end{pmatrix}w.
\end{equation}

\bigskip


\subsection{Koornwinder polynomials over the parabolic biangle}

For $\alpha,\beta>-1$, the polynomials
\begin{equation*}
P_{n,m}(x,y) = P_{n-m}^{(\alpha,\beta+m+1/2)}(2x-1)\,x^{m/2}\,P_m^{(\beta,\beta)}\left(\frac{y}{x} \right),\quad 0\le m\le n,
\end{equation*}
are orthogonal polynomials associated with the Koornwinder weight function
\begin{equation*}
w(x,y)=(1-x)^{\alpha}(x-y^2)^{\beta},
\end{equation*}
on the parabolic biangle
\begin{equation*}
\Omega=\{(x,y)\in\mathbb{R}^2:\,y^2<x<1\},
\end{equation*}
with boundary
$$
\partial \Omega=\{x-y^2=0, 0\le x\le 1\}\cup\{1-x=0, -1\le y\le 1 \}.
$$
These polynomials are obtained from the Koornwinder construction with
\begin{align*}
&w_1(x)=(1-x)^{\alpha}x^{\beta},\quad 0\le x\le1,\\
&w_2(y)=(1-y^2)^{\beta}, \quad  -1\le y\le 1,\\
&\rho(x)=\sqrt{x}.
\end{align*}
Since $\phi_1(x)=(1-x)x,\,\phi_2(y)=1-y^2$, equation (\ref{eq30}) reads
\begin{equation*}
\begin{pmatrix}
(1-x)x & \frac{1}{2}(1-x)y \\
0 & x-y^2
\end{pmatrix}\,\nabla w = \begin{pmatrix}
\beta-(\alpha+\beta)x \\
-2\beta y
\end{pmatrix}w.
\end{equation*}
A suitable choice for the simmetrization matrix is
\begin{equation*}
S=\begin{pmatrix}
1 & 0 \\
\dfrac{y}{2x} & -\dfrac{1}{4x}
\end{pmatrix},
\end{equation*}
and the resulting Pearson equation is
\begin{equation*}
\begin{pmatrix}
(1-x)x & \frac{1}{2}(1-x)y\\
\frac{1}{2}(1-x)y & \frac{1}{4}(1-y^2)
\end{pmatrix} \nabla w=\begin{pmatrix}
\beta-(\alpha+\beta)x\\
-\frac{1}{2}(\alpha+\beta)y
\end{pmatrix}w.
\end{equation*}
In this case, the associate second order linear partial differential operator is 
$$
\mathcal{L}[\cdot]=2(1-x)x\partial_{xx}+2(1-x)y\partial_{xy}+\frac{1}{2}(1-y^2)\partial_{yy}
+[2\beta+3-(2\alpha+2\beta+5)x]\partial_x-(\alpha+\beta+2)y\partial_y,
$$
and the corresponding second order linear partial differential equation satisfied by the sequence of bivariate polynomials
is
$$\mathcal{L}[P_{n,m}]= -\left[(n-m)(2n+2\alpha+2\beta+5)+\frac{1}{2}m(m+2\alpha+2\beta+3)\right]P_{n,m}.
$$
If the decomposition method is used, then a suitable choice of auxiliary functions is
\begin{equation*}
a(x,y)=\frac{2x}{x-y^2}, \quad b(x,y)=\frac{y}{x-y^2}, \quad c(x,y)=\frac{1-y^2}{2(1-x)(x-y^2)},
\end{equation*}
and we obtain again the same matrix Pearson equation. Notice that the Koornwinder polynomials over
the parabolic biangle are classical.


\bigskip

\subsection{Koornwinder polynomials over the triangle}

Following \cite{Koor75}, for $\alpha, \beta, \gamma>-1$ these polynomials correspond to
\begin{align*}
&w_1(x)=(1-x)^{\alpha}x^{\beta+\gamma}, \quad 0\le x\le1,\\
&w_2(y)=(1-y)^{\beta}y^{\gamma}, \quad 0\le y\le 1,\\
&\rho(x)=x,
\end{align*}
on the triangle
\begin{equation*}
\mathbf{T}=\{(x,y)\in\mathbb{R}^2:\,0<y<x<1\}.
\end{equation*}
The polynomials
\begin{equation*}
P_{n,m}(x,y)=P_{n-m}^{(\alpha,\beta+\gamma+2m+1)}(2x-1)\,x^m\,P_m^{(\beta,\gamma)}\left(\frac{2y}{x}-1\right), \quad 0\le m\le n,
\end{equation*}
are orthogonal with respect to the weight function
\begin{equation*}
w(x,y)=(1-x)^{\alpha}(x-y)^{\beta}y^{\gamma}.
\end{equation*}
Notice that $\phi_1(x)=\phi_2(x)=(1-x)x$, and (\ref{eq30}) reads
\begin{equation*}
\begin{pmatrix}
(1-x)x & (1-x)y \\
0 & (x-y)y
\end{pmatrix}\,\nabla w = \begin{pmatrix}
\beta+\gamma-(\alpha+\beta+\gamma)x \\
\gamma x-(\beta+\gamma)y
\end{pmatrix}\,w.
\end{equation*}
This matrix equation is symmetrized by left multiplication times
\begin{equation*}
S =
\begin{pmatrix}
1 & 0\\
\dfrac{y}{x} & \dfrac{1}{x}
\end{pmatrix},
\end{equation*}
and the resulting Pearson equation is
\begin{equation*}
\begin{pmatrix}
(1-x)x & (1-x)y \\
(1-x)y & (1-y)y
\end{pmatrix}\nabla w=\begin{pmatrix}
\beta+\gamma-(\alpha+\beta+\gamma)x\\
\gamma-(\alpha+\beta+\gamma)y
\end{pmatrix}w.
\end{equation*}
The second order linear partial differential equation satisfied by the Koornwinder polynomials over the triangle is
$$
\mathcal{L}[P_{n,m}]= -n(n+\alpha+\beta+\gamma+2) P_{n,m},$$
where
\begin{eqnarray*}
\mathcal{L}[\cdot] &=& (1-x)x\partial_{xx}+2(1-x)y\partial_{xy}+(1-y)y\partial_{yy}\\
&~&
+[\beta+\gamma+2-(\alpha+\beta+\gamma+3)x]\partial_x+[\gamma+1-(\alpha+\beta+\gamma+3)y]\partial_y.
\end{eqnarray*}
If the decomposition method is used, we get the same equation by choosing the auxiliary functions as
\begin{equation*}
a(x,y)=\frac{x}{(x-y)y}, \quad b(x,y)=\frac{1}{x-y}, \quad c(x,y)=\frac{1-y}{(1-x)(x-y)}.
\end{equation*}
Observe that the Koornwinder polynomials over the triangle are classical.


\subsection{Laguerre--Jacobi Koornwinder polynomials}

In \cite{FPP12} some new examples of Koornwinder bivariate weight functions
were introduced. This two examples are studied here.

Consider the Laguerre and Jacobi weight functions in one variable
\begin{align*}
&w_1(x)=x^{\alpha}e^{-x}, \quad 0\le x<\infty, \quad \alpha>-1,\\
&w_2(y)=(1-y)^{\beta}, \quad -1\le y \le 1,\quad \beta>-1.
\end{align*}
The polynomials
\begin{equation*}
P_{n,m}(x,y) = L_{n-m}^{(\alpha+2m+1)}(x)\,x^m\,P_m^{(0,\beta)}\left( \frac{y}{x}\right), \quad 0\le m\le n,
\end{equation*}
are orthogonal with respect to
\begin{equation*}
w(x,y)=x^{\alpha-\beta}e^{-x}(x-y)^{\beta},
\end{equation*}
defined on the unbounded region
\begin{equation*}
\Omega=\{(x,y)\in\mathbb{R}^2:\,-x<y<x,\,x>0\}.
\end{equation*}
Here $\phi_1(x)=x,\,\phi_2(y)=1-y^2$, and thus (\ref{eq30}) reads
\begin{equation*}
\begin{pmatrix}
x & y\\
0 & x^2-y^2
\end{pmatrix}\,\nabla w= \begin{pmatrix}
\alpha-x\\ -\beta(x+y)
\end{pmatrix}\,w.
\end{equation*}
Multiplying this equation by the symmetrization matrix
\begin{equation*}
S =
\begin{pmatrix}
1 & \dfrac{1}{x+y}\\
\\
1 & 1+\dfrac{1}{x+y}
\end{pmatrix},
\end{equation*}
we get the following Pearson equation for $w$
\begin{equation*}
\begin{pmatrix}
x & x\\
x & x^2-y^2+x
\end{pmatrix}\nabla w=\begin{pmatrix}
\alpha-\beta-x\\ -\beta(x+y)+(\alpha-\beta-x)
\end{pmatrix}w.
\end{equation*}
A suitable choice of auxiliary functions for the decomposition method is
\begin{equation*}
a(x,y)=\frac{1}{x^2-y^2}, \quad b(x,y)=\frac{1}{x^2-y^2},\quad c(x,y)=\frac{x^2-y^2+y}{x(x^2-y^2)},
\end{equation*}
and the resulting Pearson equation is
\begin{equation*}
\begin{pmatrix}
x & x\\
x & x^2-y^2+y
\end{pmatrix}\nabla w=\begin{pmatrix}
\alpha-\beta-x\\ -\beta(x+y)+\alpha-x
\end{pmatrix}w.
\end{equation*}
The Laguerre--Jacobi Koornwinder polynomials satisfy the difference--differential equation
$$\mathcal{L}[P_{n,m}] = \lambda_{n,m}\,P_{n,m}+\lambda_{n,m-1}\,P_{n,m-1}+\lambda_{n,m-2}\,P_{n,m-2},
$$
where
$$
\mathcal{L}[\cdot] =
x\partial_{xx}+2x\partial_{xy}+(x^2-y^2+x)\partial_{yy}
+(1+\alpha-\beta-x)\partial_x+[\alpha-\beta+1-(1+\beta)x-(2+\beta)y]\partial_y,
$$
and
\begin{align*}
&\lambda_{n,m}=-n-m(m+\beta),\\
&\lambda_{n,m-1}=-(m-1)(\beta+1),\\
&\lambda_{n,m-2}=m(m-1).
\end{align*}
Notice that Laguerre--Jacobi Koornwinder polynomials are classical according to Theorem \ref{teor-clas}.

\bigskip


\subsection{Laguerre--Laguerre Koornwinder polynomials}

In \cite{FPP12} the Laguerre weight functions in one variable were considered
\begin{align*}
&w_1(x)=x^{\alpha}e^{-x}, \quad 0\le x<\infty, \quad \alpha>-1,\\
&w_2(y)=y^{\beta}e^{-y}, \quad 0\le y<\infty, \quad \beta>-1,\\
&\rho(x)=x, \quad \alpha-\beta >-1,
\end{align*}
then Laguerre--Laguerre Koornwinder polynomials defined by
\begin{equation*}
P_{n,m}(x,y)=L^{(\alpha+2m+1)}_{n-m}(x)\,x^m\,L^{(\beta)}_m\left(\frac{y}{x}\right), \quad 0\le m\le n,
\end{equation*}
are orthogonal with respect to the weight function
\begin{equation*}
w(x,y)=x^{\alpha-\beta}y^{\beta}e^{-(x+y/x)},
\end{equation*}
on the unbounded region $\Omega=[0,\infty)\times[0,\infty)$.
Here, equation (\ref{eq30}) reads
\begin{equation*}
\begin{pmatrix}
x & y \\
0 & xy
\end{pmatrix}\,\nabla w=\begin{pmatrix}
\alpha-x \\ (\beta+1)x-y
\end{pmatrix}\, w,
\end{equation*}
and a suitable symmetrization matrix is
\begin{equation*}
S =
\begin{pmatrix}
x & 0\\
y & 1
\end{pmatrix}.
\end{equation*}
A convenient choice of auxiliary equations for the decomposition method is
\begin{equation*}
a(x,y)=\frac{1}{xy}, \quad b(x,y)=\frac{1}{x^2}, \quad c(x,y)=\frac{x+y}{x^3}.
\end{equation*}
The resulting Pearson equation for $w$ is
\begin{equation*}
\begin{pmatrix}
x^2 & xy \\
xy & (x+y)y
\end{pmatrix}\nabla w=\begin{pmatrix}
(\alpha-x)x \\ (\alpha-1)y+\beta x-xy
\end{pmatrix}w.
\end{equation*}
From equation (\ref{eq35}) for this case, we conclude that $w$ also satisfies the Pearson equation
\begin{equation*}
\begin{pmatrix}
x^2 & 0\\
0 & xy
\end{pmatrix}\nabla w=\begin{pmatrix}
(\alpha-\beta-1-x)x+y \\ (\beta+1)x-y
\end{pmatrix} w,
\end{equation*}
and the Laguerre-Laguerre Koornwinder polynomials satisfy the difference--differential relation
$$
\mathcal{L}[P_{n,m}] = \lambda_{n+1,m}P_{n+1,m} + \lambda_{n,m+1} P_{n,m+1} + \lambda_{n,m} P_{n,m}+\lambda_{n,m-1}P_{n,m-1},
$$
where
$$\mathcal{L}[\cdot] =
x^2\partial_{xx} + x y\partial_{yy} + [(\alpha-\beta+1-x)x+y]\partial_x+[(\beta+2)x-y]\partial_y,
$$
and \begin{align*}
&\lambda_{n+1,m}=-(n-m),\\
&\lambda_{n,m+1}=n-m+m(m-1),\\
&\lambda_{n,m}=(n-m)(n-m+\alpha+\beta)-m,\\
&\lambda_{n,m-1}=(m-1)(\beta+2).
\end{align*}
Observe that the Laguerre--Laguerre Koornwinder weight satisfy a matrix Pearson equation with
$\deg\Phi =\deg\Psi =2$, and they are semiclassical.

\end{document}